\documentclass[11pt, oneside]{article}   	
\usepackage{geometry}                		
\usepackage{graphicx}				
\usepackage{amssymb}
\usepackage{amsmath}
\usepackage{amsthm}
\usepackage{enumerate}
\usepackage{tikz-cd}
\usepackage{mathrsfs}
\usepackage{bbm}
\usepackage{cite}

\newtheorem{thm}{Theorem}[section]
\newtheorem{prop}[thm]{Proposition}

\newtheorem{lem}[thm]{Lemma}
\newtheorem{dfn}{Definition}
\newtheorem{rmk}{Remark}[section]

\newtheorem{conv}{Convention}
\newtheorem{for}{Formula}

\DeclareMathOperator{\Sym}{Sym}

\DeclareMathOperator{\ev}{ev}

\DeclareMathOperator{\sgn}{sgn}

\DeclareMathOperator{\Spec}{Spec}

\DeclareMathOperator{\Gr}{Gr}
\DeclareMathOperator{\ind}{ind}
\DeclareMathOperator{\Jac}{Jac}
\DeclareMathOperator{\Tr}{Tr}
\DeclareMathOperator{\rk}{rk}
\DeclareMathOperator{\codim}{codim}
\DeclareMathOperator{\Lines}{Lines}
\DeclareMathOperator{\mult}{mult}

\title{A quadratically enriched count of lines on a degree 4 del Pezzo surface.}
\author{Cameron Darwin}
\date{}							

\begin{document}
\maketitle

\abstract{
Over an algebraically closed field $k$, there are 16 lines on a degree 4 del Pezzo surface, but for other fields the situation is more subtle. In order to improve enumerative results over perfect fields, Kass and Wickelgren introduce a method analogous to counting zeroes of sections of smooth vector bundles using the Poincar{\'e}-Hopf theorem in \cite{index}. However, the technique of Kass-Wickelgren requires the enumerative problem to satisfy a certain type of orientability condition. The problem of counting lines on a degree 4 del Pezzo surface does not satisfy this orientability condition, so most of the work of this paper is devoted to circumventing this problem. We do this by restricting to an open set where the orientability condition is satisfied, and checking that the count obtained is well-defined, similarly to an approach developed by Larson and Vogt in \cite{larsonvogt}.
}

\section{Introduction}

\begin{conv}
Throughout, we will assume that $k$ is a perfect field of characteristic not equal to 2. In statements of propositions, this will be explicitly reiterated when needed.
\end{conv}

There are 16 lines on a smooth degree 4 del Pezzo surface $\Sigma$ over an algebraically closed field $k$ of characteristic not equal to 2---that is to say, there are 16 linear embeddings $\mathbb{P}^1_k \to \Sigma$ up to reparametrization. When $k$ is not algebraically closed, the situation is more subtle. For starters, one must allow ``lines'' to include linear embeddings $\mathbb{P}^1_{k'} \to \Sigma$, for finite extensions $k'/k$. Moreover, there may not be 16 such embeddings. To see why, it is useful to recall how the count is done.

A common strategy for solving enumerative problems is linearization---that is, one seeks to express the solution set as the zero locus of a section of a vector bundle $E$ over some ambient moduli space $X$. In the case of counting lines on a degree 4 del Pezzo, $X$ is $\Gr_k(2,5)$, the Grassmannian of lines in $\mathbb{P}^4_k$, and $E$ is $\Sym^2(S^\vee)\oplus\Sym^2(S^\vee)$, where $S$ is the canonical subplane bundle over $\Gr_k(2,5)$.

$\Sigma$ can be written as the complete intersection of two quadrics $f_1$ and $f_2$ in $\mathbb{P}^4$ (pg. 100 of \cite{wittenberg}). Composing a line $\mathbb{P}^1_{k'} \to S$ with the embedding $\Sigma = Z(f_1, f_2) \to \mathbb{P}^4_k$ determines a linear embedding $\mathbb{P}^1_{k'} \to \mathbb{P}^4_k$, which can itself be identified with a closed point in $\Gr_k(2,5)$ with residue field $k'$. To identify which closed points in $\Gr_k(2,5)$ correspond to lines on $\Sigma$, one notices that for each line in $\mathbb{P}^4_k$, i.e. each linear embedding $L : \mathbb{A}^2_{k'} \to \mathbb{A}^5_k$, $f_1$ and $f_2$ pull back to degree 2 polynomials on $\mathbb{A}^2_{k'}$, i.e. to elements of $\Sym^2(S_L^\vee)$. Thus $f_1$ and $f_2$ determine two sections, $\sigma_1$ and $\sigma_2$ respectively, of $\Sym^2(S^\vee)$, and the set of lines on $\Sigma$ is precisely the zero locus $Z(\sigma_1 \oplus \sigma_2)$.

For general $f_1$ and $f_2$, $Z(\sigma_1 \oplus \sigma_2)$ consists of finitely many closed points (Theorem 2.1 of \cite{debarremanivel}). The most na{\"i}ve count of lines on $\Sigma$---a literal count of the number of linear embeddings $\mathbb{P}^1_{k'} \to \Sigma$---would simply be $\#Z(\sigma_1 \oplus \sigma_2)$, but this number does not always come out to 16. To achieve an invariant answer, one could weight the lines on $\Sigma$ by the degree of the field extension $\kappa(L)/k$, and then one would have that
\[
\sum_{L \in Z(\sigma_1 \oplus \sigma_2)} [\kappa(L):k] = 16.
\]
However, this is not a genuine improvement of the count for algebraically closed $k$:

Fix an algebraic closure $\overline{k}$ of $k$. Then $\overline{X} := \Gr_{\overline{k}}(2,5)$ is the base change of $X$ from $k$ to $\overline{k}$, and $\overline{E} := \Sym^2(\overline{S}^\vee)\oplus\Sym^2(\overline{S}^\vee)$ (where $\overline{S}$ is the canonical subplane bundle over $\Gr_{\overline{k}}(2,5)$) is the base change of $E$ from $k$ to $\overline{k}$. Letting $\overline{f}_1$ and $\overline{f}_2$ denote the base changes of $f_1$ and $f_2$, the section $\overline{\sigma}_1 \oplus \overline{\sigma}_2$ of $\overline{X}$ corresponding to $\overline{f}_1$ and $\overline{f}_2$ as described earlier, is itself the base change of $\sigma_1 \oplus \sigma_2$. Moreover, the zero locus $\overline{\Sigma} = Z(\overline{f}_1, \overline{f}_2)$ is a smooth degree 4 del Pezzo over $\overline{k}$, and hence the zero locus of $\overline{\sigma}_1 \oplus \overline{\sigma}_2$ consists precisely of the lines on $\overline{\Sigma}$, of which there are 16.

To prove that the weighted sum of lines on $\Sigma$ is 16, one considers the fact that $Z(\overline{\sigma}_1 \oplus \overline{\sigma}_2)$ is the base change of $Z(\sigma_1 \oplus \sigma_2)$. Considering the base change projection
\[
c : Z(\overline{\sigma}_1 \oplus \overline{\sigma}_2) \to Z(\sigma_1 \oplus \sigma_2),
\]
one has that, for each $L \in Z(\sigma_1 \oplus \sigma_2)$, that $[\kappa(L) : k] = \#c^{-1}(L)$, and consequently
\[
\sum_{L \in Z(\sigma_1 \oplus \sigma_2)} [\kappa(L):k] = \sum_{L \in Z(\sigma_1 \oplus \sigma_2)} \#c^{-1}(L) = \# Z(\overline{\sigma}_1 \oplus \overline{\sigma}_2) = 16.
\]

Thus, while weighting the lines on $\Sigma$ by $[\kappa(L) : k]$ achieves a consistent count of 16, this is really nothing more than the original count that there are 16 lines on a smooth degree 4 del Pezzo surface over an algebraically closed field. To improve upon this count, we will use an approach introduced by Kass and Wickelgren in \cite{index} to count lines on smooth cubic surface:

Consider for a moment the classical case of a vector bundle $E$ of rank $r$ over a smooth closed manifold $X$ of dimension $r$, and consider a section $s$ of $E$ with only isolated zeroes. One might ask whether the number of zeroes of $s$ can change as $s$ is changed by a small homotopy. The answer, of course, is yes. If one studies how this can happen, one discovers two phenomena: a single zero can split into multiple zeroes, or two zeroes can cancel each other out. The former problem is analogous to the situation of a solution to an enumerative problem over $k$ splitting into multiple solutions over a field extension $k'/k$. To account for this problem, one can define a local multiplicity:

\begin{dfn}[local multiplicity]\label{mult}
Let $E$ be a smooth rank $r$ vector bundle over a smooth, closed manifold $X$ of dimension $r$. Let $s$ be a section of $E$ and $z$ an isolated zero of $s$. By choosing an open $r$-ball around $z$ and trivializing $E$ over that ball, one obtains a map $\mathbb{R}^r \to \mathbb{R}^r$ which vanishes only at 0, hence inducing a map $S^r \to S^r$ whose degree is well-defined up to a sign. Define the local multiplicity at $z$ to be the absolute value of this degree, which we will denote $\mult_z s$.
\end{dfn}

In some sense, the local multiplicity at $z$ is the ``expected'' number of zeroes $z$ will split into if $s$ is homotoped to be transversal to the zero section. Consequently, one might hope that counting local multiplicities is sufficient, in the sense that the sum
\[
\sum_{z \in Z(s)} \mult_z s
\]
is independent of $s$.

However, this does not deal with the possibility of two zeroes canceling each other out: for a section $s$ of $E$ which is already transversal to the zero section, every zero has multiplicity 1 (in the sense of Definition \ref{mult}), and hence weighting zeroes by their multiplicity simply obtains the set theoretic size of the zero set of $s$---but, as is well known, this number is still not well-defined.

The upshot of this discussion is that there is a way to weight the zeroes of a section of a smooth vector bundle which is defined purely in terms of local data, namely the multiplicity, which is analogous to weighting zeroes by the degree of the extension $\kappa(z)/k$. In the algebraic case, the latter weighting does give a well-defined count, although an unsatisfying one, while in the topological case, it does not even give a well-defined count. 

Now we will recall how the problem of giving a well-defined count is solved in the topological case, in order to motivate, by analogy, Kass-Wickelgren's approach to giving a more nuanced count in the algebraic case:

\begin{dfn}[orientation]
Let $V$ be a real vector space. Then we will think of an orientation on $V$ as a choice of a positive half of $\det V$. More generally, for a vector bundle $E$, if removing the zero section disconnects the total space of $\det E$, then an orientation on $\det E$ is a choice of a positive half of $\det E \smallsetminus \{zero\ section\}$. Note that this is equivalent to trivializing $\det E$.
\end{dfn}

The topological problem is classically solved by making an orientability assumption on $E$ and $X$. In the simplest case, one assumes that both $E$ and $X$ are oriented. Then the differential $ds$ induces a well-defined isomorphism $T_z X \to E_z$ at every zero $z$ of $s$, and $z$ can be given a sign $\sgn_zs \in \{\pm 1\}$ according to whether $ds$ preserves orientation or reverses orientation. The Poincare-Hopf theorem then says that the sum
\[
\sum_{z \in Z(s)} \sgn_zs
\]
is independent of the section $s$.

The calculation of the local signs $\sgn_zs$ is both straight-forward and informative: an orientation on $X$ induces an orientation of $T_zX$, and an orientation of $E$ induces an orientation of $E_z$. Now one can choose a neighborhood $U$ containing $z$ and coordinates $\{u^i\}$ on $U$ so that
\[
\frac{\partial}{\partial u^1} \wedge \cdots \wedge \frac{\partial}{\partial u^r}
\]
is in the positive half of $\det T_z X$. Next, one chooses a trivialization $\{e_j\}$ of $E|_U$ so that
\[
e_1 \wedge \cdots \wedge e_r
\]
is in the positive half of $\det E_z$. Together, these express $\sigma|_U$ as a map $\{f^i\} : \mathbb{R}^r \to \mathbb{R}^r$ which has a zero at $z$. The determinant of the Jacobian matrix of first partial derivatives
\[
\left( \frac{\partial f^i}{\partial u^j}\right)
\]
at $z$, which we will denote $\Jac_z (\sigma; u,e)$, depends on the choice of coordinates $\{u^i\}$, and on the trivialization $\{e_j\}$, but its sign does not. One then computes that
\[
\sgn_z s = \left\{
\begin{array}{lcl}
+1 & \ \ \ \ \ & \Jac_z(s;u,e) \sigma > 0 \\
-1 & \ \ \ \ \ & \Jac_z(s;u,e) \sigma < 0
\end{array}
\right..
\]

Unpacking this a bit more, we should note that counting the sign of the determinant has a rather straightforward homotopical interpretation: consider any linear isomorphism $\phi : \mathbb{R}^r \to \mathbb{R}^r$. Considering $S^r$ as the one point compactification of $\mathbb{R}^r$, $\phi$ determines a homeomorphism $\widetilde{\phi} : S^r \to S^r$, and it is precisely the sign of $\det \phi$ which determines the homotopy class of $\widetilde{\phi}$. Moreover, the identification of the sign of $\Jac_z(s;u,e)$ with a homotopy class of maps $S^r \to S^r$ underlies a rather direct approach to proving the Poincare-Hopf theorem, and is also an easy way to motivate the approach taken by Kass and Wickelgren:

Stably, a homotopy class of self-homeomorphisms of a sphere corresponds to an element of $\pi_0^S$, which is isomorphic to $\mathbb{Z}$. In the stable motivic homotopy category over $k$, $\pi^S_0$ is isomorphic to $GW(k)$, the Grothendieck-Witt group\footnote{More precisely, $GW(k)$ is obtained by beginning with the semiring of isomorphism classes of symmetric non-degenerate bilinear forms over $k$, with tensor product as multiplication and direct sum as addition, and group-completing the addition.} of isomorphism classes of symmetric non-degenerate bilinear forms over $k$ \cite{morel}. 

An explicit description of $GW(k)$ in terms of generators and relations can be given (this is Lemma 2.9 of \cite{algtop}; see \cite{mh} Ch. III.5 for discussion), which it will be convenient for us to record:
\begin{prop}\label{presentation}
Let $k$ be a field with characteristic not equal to 2, and consider the abelian group $GW^{pr}(k)$ generated by symbols $\langle a \rangle $ for all $a \in k^\times$ subject to the relations
\begin{enumerate}[i.]
\item $\langle uv^2 \rangle  = \langle u \rangle$
\item $ \langle u \rangle + \langle - u \rangle = \langle 1 \rangle + \langle -1 \rangle $
\item $\langle u \rangle  + \langle v \rangle = \langle u + v \rangle  + \langle (u + v)uv \rangle$ if $u + v \neq 0$
\end{enumerate}
$GW^{pr}(k)$ becomes a ring under the multiplication $\langle u \rangle \cdot \langle v \rangle = \langle uv \rangle$, and sending $\langle a \rangle$ to the bilinear form $k \otimes k \to k$ given by $x \otimes y \mapsto axy$ extends to a ring isomorphism $GW^{pr}(k) \to GW(k)$. We will implicitly assume this identification, and simply use $\langle a\rangle$ to refer to the corresponding bilinear form.
\end{prop}

Now consider a linear isomorphism $\psi : k^r \to k^r$. In the motivic homotopy category, this determines a map $\widetilde{\psi} : \mathbb{P}^r_k/\mathbb{P}^{r-1}_k \to \mathbb{P}^r_k/\mathbb{P}^{r-1}_k$, analogously to how a linear isomorphism $\mathbb{R}^r \to \mathbb{R}^r$ determined a map $S^r \to S^r$. Moreover, motivically, $\mathbb{P}^r_k/\mathbb{P}^{r-1}_k$ is a sphere, and hence the homotopy class of $\widetilde{\psi}$ represents an element of $GW(k)$, which turns out to precisely be the rank one bilinear form $\langle \det \psi \rangle$.

Viewed this way, the isomorphism class $\langle \det ds \rangle$ is the motivic analog of the sign of the determinant $\det ds$, at least when used to assign a local index to a zero of a section of a vector bundle\footnote{And also note that the multiplicative group of rank one non-degenerate bilinear forms over $\mathbb{R}$ is precisely the group of signs, i.e. the multiplicative group $\{\pm 1\}$}. In \cite{index}, Kass and Wickelgren use this idea to develop a fairly broad technique for counting zeroes of vector bundles over smooth schemes. Underlying their technique is the following orientability requirement:

\begin{dfn}[relative orientation]
Let $p : X \to \Spec k$ be a smooth scheme, and $E$ a vector bundle over $X$. Then $E$ is said to be relatively orientable if there is an isomorphism
\[
\rho : \det E \otimes \omega_{X/k} \to L^{\otimes 2}
\]
for some line bundle $L$ over $X$. The isomorphism $\rho$ is called a relative orientation, and the pair $(E, \rho)$ will be called a relatively oriented vector bundle.
\end{dfn}

Now continuing the notation in the statement of the definition, and assuming that $\rk E = \dim X = r$, suppose $s$ is a section of $E$ whose zero locus consists of finitely many closed points. Consider some zero $z$ of $s$, and suppose that there is a neighborhood $U$ of $z$ and an isomorphism $u: U \cong \mathbb{A}^r_k$ (or an isomorphism with an open subset of $\mathbb{A}^r_k$). Note that the coordinate vector fields on $\mathbb{A}^r_k$ determine a basis  $\{\partial_{u_1}|_z, \ldots, \partial_{u_r}|_z\}$ for $(T_X)_z$.

Next, suppose that there is a trivialization of $E|_U$ by sections $\{e_1, \ldots, e_r\}$ such that the map $\det (T_X)_z \to \det E_z$ defined by
\[
\partial_{u_1}|_z \wedge \cdots \wedge \partial_{u_r}|_z\longmapsto e_1 \wedge \cdots \wedge e_r
\]
is a square in $(\omega_X)_z \otimes \det E_z \cong (L_z)^{\otimes 2}$. Then we make the following definiton:

\begin{dfn}[good parametrization]
In the notation of the preceding paragraphs, and the conditions described, suppose also that the map $s_{u,e}:\mathbb{A}^r_k \to \mathbb{A}^r_k$ corresponding to $s$ over $U$ is {\'e}tale at $z$. Then we will refer to the coordinates $u: U \to \mathbb{A}^r_k$ (allowing this notation to also include the case of an isomorphism between $U$ and an open subset of $\mathbb{A}^r_k$) and the trivialization $\{e_1, \ldots, e_r\}$ of $E|_U$ together as a good parametrization near $z$.
\end{dfn}

Continuing with the same notation and assumptions, we consider two cases: first, suppose $z$ is $k$-rational, i.e. $\kappa(z) = k$. Then evaluating the Jacobian matrix $\left(\frac{\partial (s_{u,e})_i}{\partial u_j}\right)$ at $z$ yields a matrix of elements of $k$. This matrix has a determinant in $k$, which depends, as in the case of a section of a vector bundle over a manifold, on the choice of coordinates and trivialization. However, again analogous to the classical case, Kass and Wickelgren show in \cite{index} that provided that a good parametrization is used to compute the determinant, the bilinear form
\[
\left \langle \det \left(\frac{\partial (s_{u,e})_i}{\partial u_j}\right) \right \rangle
\]
is well-defined up to isomorphism.

When $z$ is not $k$-rational, we need to work a bit harder. Evaluating the Jacobian matrix $\left(\frac{\partial (s_{u,e})_i}{\partial u_j}\right)$ at $z$ on the nose yields a matrix of linear maps $\kappa(z) \to k$. However, by base changing the map $s_{u,e}$ to a map $s'_{u,e} : \mathbb{A}^r_{\kappa(z)} \to \mathbb{A}^r_{\kappa(z)}$ and then evaluating at $z$ one obtains a matrix $\left(\frac{\partial (s'_{u,e})_i}{\partial u'_j}\right)$ of elements of $\kappa(z)$, and this matrix now has a determinant in $\kappa(z)$. We would like to try to use the bilinear form
\[
\left \langle \det \left(\frac{\partial (s'_{u,e})_i}{\partial u'_j}\right) \right \rangle
\]
to define our local sign, but we immediately run into the problem that this is a bilinear form over $\kappa(z)$, not over $k$.

If we make the additional assumption that $\kappa(z)/k$ is separable---which is automatically guaranteed if, for example, $k$ is perfect---then we can use the trace map $\Tr_{\kappa(z)/k} : \kappa(z) \to k$. This map is surjective, and hence for any vector space $V$ over $\kappa(z)$, and any non-degenerate symmetric bilinear form $b : V \otimes V \to \kappa(z)$, composing $b$ with $\Tr_{\kappa(z)/k}$ and viewing $V$ as a vector space over $k$ produces a non-degenerate symmetric bilinear form $\Tr_{\kappa(z)/k} b$.

In \cite{index}, Kass and Wickelgren show that, provided that a good parametrization is used, the bilinear form
\[
\Tr_{\kappa(z)/k} \left \langle \det \left(\frac{\partial (s'_{u,e})_i}{\partial u'_j}\right) \right \rangle
\]
is well-defined. Moreover, this recovers the same bilinear form that would have been defined if $z$ were $k$-rational, because $\Tr_{k/k}$ is the identity map. Consequently, we make the following definition:

\begin{dfn}[Jacobian form]\label{jacform}
Let $(E,\rho)$ be a relatively oriented vector bundle over a smooth scheme $X \to \Spec k$ for $k$ a perfect field, and assume that $\rk E = \dim X = r$. Let $s$ be a section of $E$ whose zero locus consists of finitely many closed points. Assume also that there is a good parametrization at every zero $z$ of $s$. Then we define the Jacobian form 
\[
\Tr_{\kappa(z)/k} \langle \Jac_z (s;\rho)\rangle
\]
at $z$ to be the well-defined bilinear form $k \otimes k \to k$ given by computing
\[
\Tr_{\kappa(z)/k} \left \langle \det \left(\frac{\partial (s'_{u,e})_i}{\partial u'_j}\right) \right \rangle
\]
in any good parametrization around $z$. Note that this bilinear form has rank $[\kappa(z) : k]$.
\end{dfn}

Now return to the situation of lines on a degree 4 del Pezzo surface. Then $X = \Gr_k(2,5)$ and $E = \Sym^2(S^\vee) \oplus \Sym^2(S^\vee)$, and we have that $X$ admits a cover by open sets which are isomorphic to $\mathbb{A}^6_k$. Moreover, for general $f_1$ and $f_2$, $Z(\sigma_1 \oplus \sigma_2)$ consists of finitely many closed points, and is itself {\'e}tale over $k$. For finite $k$, this can be refined to saying that there is a Zariski open subset of the space of sections of $E$ whose closed points correspond to degree 4 del Pezzos over a finite extension of $k$ where $Z(\sigma_1 \oplus \sigma_2)$ is finite {\'e}tale over $k$. Thus, for general $f_1$ and $f_2$, $\sigma_1 \oplus \sigma_2$ is a section whose zero set consists of finitely many closed points, at each of which there is a good parametrization. We would thus like to try to count lines on a del Pezzo by assigning each line its Jacobian form. But we run into a problem: $E$ is not relatively orientable.

To explain how we get around this problem, it is useful to explain why $E$ fails to admit a relative orientation:

Consider the Pl{\"u}cker embedding $X \hookrightarrow \mathbb{P}^9_k$. The Picard group of $X$ is generated by the restriction of $\mathcal{O}_{\mathbb{P}^9_k}(1)$ to $X$, which we will denote $\mathcal{O}_X(1)$. Moreover, the tautological line bundle on $\mathbb{P}^9_k$ restricts on $X$ to the determinant of $S$, so that $\det S = \mathcal{O}_X(-1)$. The tautological short exact sequence
\begin{center}
\begin{tikzcd}
0 \arrow{r}& S\arrow{r} & \mathscr{O}_X^{\oplus 5} \arrow{r}& Q \arrow{r}& 0
\end{tikzcd},
\end{center}
together with the isomorphism $T_{X/k} \cong S^\vee \otimes Q$, implies that $\omega_{X/k} = \mathcal{O}_X(-5)$.

We also have that $\det \Sym^2(S^\vee) = (\det S^\vee)^{\otimes 3}$, and hence $\det \Sym^2(S^\vee) = \mathcal{O}_X(3)$. Taken all together, we thus compute that, in the Picard group,
\[
\det E \otimes \omega_{X/k} = \mathcal{O}_X(1),
\]
and hence $E$ is not relatively orientable.

The Pl{\"u}cker embedding exhibits the zero locus of $\sigma_1 \oplus \sigma_2$ as closed points in $\mathbb{P}^9_k$. Provided that $|k| > 16$, we will show (Proposition \ref{nondeg one form}) that there is a section $s$ of $\mathcal{O}_{\mathbb{P}^9_k}(1)$, and hence a corresponding section of $\mathcal{O}_X(1)$, whose zero locus is disjoint from $Z(\sigma_1 \oplus \sigma_2)$. 

\begin{dfn}[non-degenerate on lines]\label{nondeg}
We will refer to a section $s$  of $\mathcal{O}(1)$ whose zero locus is disjoint from $Z(\sigma_1 \oplus \sigma_2)$ as a ``one form\footnote{Our terminology ``one form'' refers not to K{\"a}hler one forms, but to the fact that a section of $\mathcal{O}_{\mathbb{P}^n_k}(1)$ corresponds to a one form on $\mathbb{A}^{n+1}_k$, i.e. a degree one homogeneous polynomial} non-degenerate on lines.''
\end{dfn}

Letting $U$ denote the complement of $Z(s)$ in $X$, the fiber-wise map
\[
\alpha \oplus \beta \mapsto s \otimes \alpha \oplus s^{\otimes 2} \otimes \beta
\]
determines an isomorphism between $E|U$ and the restriction of
\[
\widetilde{E} := \mathcal{O}_X(1) \otimes \Sym^2(S^\vee) \oplus \mathcal{O}_X(2) \otimes \Sym^2(S^\vee).
\]
to $U$. By chasing through the same type of computation we used to show that $E$ is not relatively orientable, but this time for $\widetilde{E}$, we obtain a canonical relative orientation $\rho$ on $\widetilde{E}$. We now make the following definition:

\begin{dfn}[twisted Jacobian form]\label{twjacform}
With notation as in the preceding paragraphs, consider some $z \in Z(\sigma_1 \oplus \sigma_2)$, and let $\widetilde{\sigma}$ denote the section
\[
s \otimes \sigma_1 \oplus s^{\otimes 2} \otimes \sigma_2.
\]
We define
\[
\Tr_{\kappa(z)/k} \langle \widetilde{\Jac}_z (f_1,f_2; s)\rangle := \Tr_{\kappa(z)/k} \langle \Jac_z (\widetilde{\sigma}; \rho)\rangle, 
\]
where the right side is defined as in Definition \ref{jacform}
\end{dfn}

We are now prepared to state our main result in the case that $|k| > 16$:

\begin{thm}[Main Result]
Let $\Sigma = Z(f_1,f_2) \subset \mathbb{P}^4_k$ be a general smooth degree 4 del Pezzo surface over a perfect field $k$ of characteristic not equal to 2, and assume that $|k| \geqslant 16$. Let $s$ be a one-form non-degenerate on the lines on $\Sigma$ (see Definition \ref{nondeg}).  Let $\Lines(\Sigma)$ denote the set of linear embeddings $\mathbb{P}^1_{k'} \to \Sigma$, where $k'$ ranges over all finite extensions of $k$. Then
\begin{equation} \label{result}
\sum_{L \in \Lines(\Sigma)} \Tr_{\kappa(L)/k} \langle \widetilde{\Jac}_L (f_1, f_2;s)\rangle = 8H,
\end{equation}
where $H = \langle 1 \rangle + \langle -1\rangle \in GW(k)$, and the summand is the twisted Jacobian form of Definition \ref{twjacform}.
\end{thm}

\begin{rmk}\label{general}
For an infinite field, this result automatically applies to infinitely many degree 4 del Pezzo surfaces over $k$. For any particular finite field, it is conceivable that the result as stated does not apply to any degree 4 del Pezzo surface over $k$. However, the proof shows that there is a Zariski open subset in $\Spec \Sym^{\bullet} \Gamma\left(\Sym^2(S^\vee) \oplus \Sym^2(S^\vee)\right)$, every closed point of which corresponds to a degree 4 del Pezzo surface over a finite extension of $k$ where equation (\ref{result}) holds. 
\end{rmk}
In the course of the paper, we will explain how to extend this result to the case where $|k| \leqslant 16$.

\section{Relation To Other Work}

In addition to the work of Kass-Wickelgren and Bachmann-Wickelgren in defining enriched Euler classes, enrichment of enumerative geometry to take values in quadratic forms has been investigated by Levine in \cite{aspects}, and enriched results under the assumption of relative orientability have been obtained for various classical problems (see, e.g., \cite{pauli}, \cite{bezout},\cite{cdh},\cite{wendt},\cite{srinivasan}). The problem of enriching enumerative results in the case $k = \mathbb{R}$ has been studied in detail by Finashin and Kharlamov, among others---see, for example \cite{real}. 

The problem of quadratically enriched counts when relative orientability fails has been dealt with by Larson and Vogt in the case of bitangents to a smooth plane quartic \cite{larsonvogt}. Larson and Vogt define the notion of ``relatively orientable relative to a divisor,'' where a vector bundle is relatively orientable away from a certain effective divisor. They then show, essentially, that counts of zeroes are homotopy invariant as long as the homotopy does not move a zero across the chosen divisor, and use this to define counts dependent on choosing a connected component of the space of sections with isolated zeroes. Our approach is similar in that we rely on restricting to an open subset where the relevant vector bundle is relatively orientable. However, rather than dealing with homotopy invariance of counts directly, we use axiomatic properties of Chow-Witt groups and the Barge-Morel Euler class to obtain a well-defined count.

\section{Acknowledgments}

Thanks to Kirsten Wickelgren for support and extensive discussions, without which this paper could not have been written. Thank you to Olivier Wittenberg for helpful comments and correspondence. The author was partly supported by NSF DMS-2001890 and by a scholarship from the Sloan Foundation.

\section{Oriented Intersection Theory}\label{oriented intersection}

In order to prove our result, we will need to give a manifestly parametrization-invariant interpretation of the Jacobian form
\[
\Tr_{\kappa(L)/k} \langle \Jac_L(\sigma;\rho)\rangle.
\]
To do this, it will prove most convenient to first recall the classical intersection theory behind counting zeroes of sections of vector bundles, for which a good reference is \cite{intersection}, particularly Chapter 6.

Let $p : X \to \Spec k$ be smooth and proper of dimension $d$, and $E$ a vector bundle over $X$ of rank $d$. Let $z :X \to E$ be the zero section, and $s : X \to E$ any section such that $Z(s) \to X$ is a regular embedding. Then consider the fiber square
\begin{center}
\begin{tikzcd}
Z(s) \arrow{r}{j}\arrow{d}{i} & X \arrow{d}{s} \\
X \arrow{r}{z} & E
\end{tikzcd}
\end{center}
Because the bottom arrow is a regular embedding of codimension $d$, the refined Gysin homomorphism $z^! : CH^*(X) \to CH^{* + d - \codim Z(s)}(Z(s))$ is defined by $z^![V] = [X\cdot V]$, where $X \cdot V$ is the intersection product. By composing $z^!$ with proper push-forward $i_* : CH^*(Z(s)) \to CH^{*+\codim Z(s)}(X)$, one obtains a homomorphism $\phi : CH^*(X) \to CH^{*+d}(X)$, which turns out not to depend on $s$. Moreover, this homomorphism can be given a rather explicit description:

Because $Z(s) \to X$ is a regular embedding, it consists of regularly embedded clopen components $j_m : Z_m \to X$, and each has a normal bundle $N_{Z_m} X$. Because the normal bundle of $X$ in $E$ along $z$ is simply $E$ itself, $N_{Z_m}X$  is naturally mapped into\footnote{The normal cone of $Z$ in $X$ will be included into $i^*N_XE$ quite generally, and in the case that $Z \hookrightarrow X$ is a regular embedding, as it is here, the normal cone is a vector bundle, namely the normal bundle. For proof, see the beginning of section 6.1 of \cite{intersection}.} $j_m^*E$ and the excess bundle $\mathcal{E}_m$ on $Z_m$ is defined by the short exact sequence
\begin{center}
\begin{tikzcd}
0 \arrow{r} & N_{Z_m}X \arrow{r} & j_m^*E \arrow{r} &  \mathcal{E}_m \arrow{r} & 0.
\end{tikzcd}
\end{center}
Now let $e(\mathcal{E}_m)$ denote the top Chern class of $\mathcal{E}_m$, considered as a homomorphism $CH^*(Z_m) \to CH^{*+\rk \mathcal{E}_m}(Z_m)$. Then the excess intersection formula computes $\phi$: for any $\alpha \in CH^*(X)$, one has
\[
\phi(\alpha) = \sum_m i_*(e(\mathcal{E}_m)(j_m^*(\alpha))).
\]

At one extreme, where $s = z$, one has that $Z(s) = X$, and $\mathcal{E} = E$, and so also, in fact, that $\phi = e(E)$. At the other extreme, when $Z(s)$ consists of simple zeroes, each $\mathcal{E}_m$ is zero-dimensional, and hence $e(\mathcal{E}_m)$ is simply the identity on $CH^*(Z_m)$. Now letting $[X]$ denote the fundamental cycle of $X$, one has that
\[
p_*(e(E)([X])) = \sum_m \ind_{z_m}(s),
\]
where we define
\[
\ind_{z_m}(s) = p_*(i_*(j_m^*([X]))) \in CH^0(\Spec k) = \mathbb{Z}.
\]

When $k$ is algebraically closed, $\ind_{z_m}(s)$ is 1 at every simple zero, and by piggy-backing off this and considering the fibers of a base change morphism as we described in the first section, one can actually compute more generally that $\ind_{z_m}(s) = [\kappa(z_m):k]$.

As we explained in the first section, we need to take orientability data into account if we wish to refine this procedure in the case that $k$ is not algebraically closed. This orientability data can be captured using the Chow-Witt groups---also known as oriented Chow groups---of Barge-Morel and Fasel (see \cite{bargemorel} and \cite{chowwitt}). To explain how, we briefly return to the topological case:

Regardless of orientability assumptions, an Euler class can be defined for a vector bundle $E$ over a topological manifold $X$, at the cost of using not ordinary cohomology, but cohomology twisted by a local system. More precisely, letting $\mathfrak{o}(E)$ denote the local system whose fiber at $x \in X$ is $H_n(E|_x, E|_x \smallsetminus \{0\}; \mathbb{Z})$, an Euler class $e(E) \in H^n(X; \mathfrak{o}(E))$ can always be defined. Moreover, when $\mathfrak{o}(E)$ is isomorphic to the orientation sheaf of $X$, twisted Poincar{\'e} duality provides an isomorphism
\[
H^n(X; \mathfrak{o}(E)) \cong H^n(X; \mathfrak{o}(X)) \to H_0(X; \mathbb{Z}),
\]
providing $E$ with an Euler number. 

Returning to the motivic case, a similar story can be told. The Chow-Witt groups of Barge-Morel and Fasel give an enriched version of the ordinary Chow groups which can be twisted by any line bundle. Moreover, for any vector bundle $E$ over a smooth $k$-scheme $X$, an Euler class $\widetilde{e}(E)$ is defined by Barge-Morel in \cite{bargemorel} and Fasel in \cite{chowwitt}, which recovers the top Chern class in the sense that the diagram
\begin{center}
\begin{tikzcd}
\widetilde{CH}^*(X) \arrow{r}{\widetilde{e}(E)}  \arrow{d} & \widetilde{CH}^{* + r}(X; \det E^\vee) \arrow{d} \\
CH^*(X) \arrow{r}{e(E)} & CH^{* + r}(X)
\end{tikzcd}
\end{center}
commutes (see \cite{chowwitt}), where the vertical arrows are a natural comparison morphism from the Chow-Witt groups $\widetilde{CH}^*$ to ordinary Chow groups. Note that the top right is an analog of the classical Euler class being valued in cohomology twisted by $\mathfrak{o}(E)$. 

The cost of being able to twist by line bundles is that pull-back and push-forward become more complicated. For our purposes, it will be sufficient to note that for a regular embedding $i : X \to Y$, and any line bundle $L$ on $Y$, there is an induced pull-back
\[
i^* : \widetilde{CH}^*(Y; L) \to \widetilde{CH}^*(X; i^*L),
\]
while for a proper morphism $f : X \to Y$ of relative dimension $d$, and any line bundle $L$ on $Y$, there is an induced push-forward
\[
f_* : \widetilde{CH}^*(X; f^*L \otimes \omega_{X/k}) \to \widetilde{CH}^{*-d}(Y; L \otimes \omega_{Y/k}).
\]

Together, these allow an oriented version of the excess intersection formula, due to Fasel in \cite{excess}, to be used. First, considering the fundamental class $[X] \in \widetilde{CH}^0(X)$, $j_m^*([X])$ is defined in $\widetilde{CH}^0(Z_m)$. Then $\widetilde{e}(\mathcal{E}_m)(j_m^*([X]))$ is defined in $\widetilde{CH}^{d - \codim Z_m}(Z_m; \det\mathcal{E}_m^\vee)$.
But now the short exact sequence
\begin{center}
\begin{tikzcd}
0 \arrow{r} & N_{Z_m}X \arrow{r} & j_m^*E \arrow{r} &  \mathcal{E}_m \arrow{r} & 0
\end{tikzcd}
\end{center}
defining the excess normal bundle, affords an isomorphism
\[
\det \mathcal{E}_m^\vee \cong j_m^*\det E^\vee \otimes \det N_{Z_m}X,
\]
while the normal short exact sequence
\begin{center}
\begin{tikzcd}
0 \arrow{r} & T_{Z_m/k} \arrow{r} & j_m^*T_{X/k} \arrow{r} &  N_{Z_m}X \arrow{r} & 0
\end{tikzcd}
\end{center}
affords an isomorphism
\[
\det N_{Z_m}X \cong j_m^*\omega_{X/k}^\vee \otimes \omega_{Z_m/k}.
\]
Consequently, $\widetilde{e}(\mathcal{E}_m)(j_m^*([X]))$ is actually defined in the Chow-Witt group 
\[
\widetilde{CH}^{d - \codim Z_m}(Z_m; j_m^*(\det E^\vee \cong \omega_{X/k}^\vee) \otimes \omega_{Z_m/k}),
\] 
and considering the fact that $i|_{Z_m} = j_m$, we then have that $i_*(\widetilde{e}(\mathcal{E}_m)(j_m^*([X]))$ is defined in $\widetilde{CH}^d(X; \det E^\vee)$, and the excess intersection formula has an oriented analog due to Fasel (see Theorem 32 of \cite{excess} and Remark 5.22 from \cite{bw}):
\begin{for}[oriented excess intersection formula]
\[
\widetilde{e}(E)([X]) = \sum_m i_*(\widetilde{e}(\mathcal{E}_m)(j_m^*([X]))).
\]
\end{for}

Now recall that in the ordinary Chow case, we were able to recover the count by pushing forward along $p_* : CH^d(X) \to CH^0(\Spec k)$. The difficulty now is the restrictions on when push-forward is defined for Chow-Witt groups. The only line bundles on $\Spec k$ are trivial, and so there is not necessarily a push-forward
\[
p_* : \widetilde{CH}^d(X; \det E^\vee) \to \widetilde{CH}^0(\Spec k)
\]
defined.

It is here where the need for relative orientability enters the story: The Chow-Witt groups are SL-orientable, meaning that for two line bundles $L'$ and $L''$ over $X$, and an isomorphism $\psi : L' \to L'' \otimes L^{\otimes 2}$ for some third line bundle $L$, there is an induced isomorphism
\[
\psi : \widetilde{CH}^*(X; L') \to \widetilde{CH}^*(X; L'').
\]
For a relatively orientable vector bundle $E$, then, the relative orientation furnishes an isomorphism $\widetilde{CH}^*(X; \det E^\vee) \to \widetilde{CH}^*(X; \omega_{X/k})$, and hence we can consider an augmented pushforward $p_*^\rho : \widetilde{CH}^*(X; \det E^\vee) \to \widetilde{CH}^{* -d}(\Spec k)$ given as the composition
\begin{center}
\begin{tikzcd}
\widetilde{CH}^*(X; \det E^\vee) \arrow{r}{\rho} & \widetilde{CH}^*(X; \det E^\vee) \arrow{r}{p_*} & \widetilde{CH}^{*-d}(\Spec k).
\end{tikzcd}
\end{center}

Now we can define the oriented local index as follows:
\begin{dfn}
Let $\rho : \det E^\vee \otimes \omega_{X/k} \to L^{\otimes 2}$ be a relative orientation on $E$, and let $s$ be a section of $E$ whose zero locus is regularly embedded $i : Z(s) \to X$. Let $z \in Z(s)$ be an isolated zero (i.e. a closed point which is itself a clopen component of $Z(s)$). Let $j_z : \{z\} \to X$ be the inclusion. Then we define the oriented index to be
\[
\ind^{or}_{z}(s;\rho):= p^\rho_*(i_*(\widetilde{e}(\mathcal{E}_z)(j_z^*([X]))))
\]
\end{dfn}

Because $\widetilde{CH}^0(\Spec k) = GW(k)$ (see \cite{algtop}), this actually assigns an index valued in $GW(k)$ to each zero. Moreover, this index agrees with the one of Kass-Wickelgren, as proven in Sections 2.3 and 2.4 of \cite{bw}, in particular Proposition 2.31:
\begin{prop}\label{oriented index}
Let $X \to \Spec k$ be smooth, and $E$ a vector bundle over $X$, with $\rho :\omega_{X/k} \otimes \det E \to L^{\otimes 2}$ a relative orientation. Let $s$ be a section, and $z$ a simple zero of $s$ admitting a good parametrization, and such that $\kappa(z)/k$ is separable (e.g. if $k$ is perfect). Then
\[
\Tr_{\kappa(z)/k} \langle \Jac_z(s;\rho)\rangle = \ind^{or}_z(s;\rho).
\]
\end{prop}

We now have the notation and definitions needed to prove our main result in the next section.

\section{Proof of Main Result}

We will first prove the following:
\begin{prop}\label{nondeg one form}
Let $|k| > 16$. Then there is a one-form non-degenerate on lines. 
\end{prop}

\noindent We will require the following lemma, which is well-known, but which we include for the sake of completeness:

\begin{lem}\label{non-vanishing}
Let $k$ be a field such that $|k| > n$, and let $z_1, \ldots, z_n$ be points in $\mathbb{P}^d_k$ for any $d$. Then there is a section $s$ of $\mathcal{O}_{\mathbb{P}^d_k}(1)$ such that $Z(s)$ does not contain any of the points $z_1, \ldots, z_n$.
\end{lem}

\begin{proof}[Proof of lemma]
Sections of $\mathcal{O}_{\mathbb{P}^d_k}(1)$ are precisely degree one homogeneous polynomials on $\mathbb{A}^{d+1}_k$, and hence $\Gamma(\mathcal{O}_{\mathbb{P}^d_k}(1))$ isomorphic as a vector space over $k$ to $(k^{d+1})^\vee$. Choose an identification. Now for each $i = 1, \ldots, n$, consider the evaluation map $\ev_i : (k^{d+1})^\vee \to \kappa(z_i)$. This is $k$-linear and non-zero, and hence $\ker \ev_i$ is a positive codimension subspace of $(k^{d+1})^\vee$. Now if $z_i \in Z(s)$ for some section $s$ of $\mathcal{O}_{\mathbb{P}^d_k}(1)$, then $s \in \ker \ev_i$.

Consequently, to find some $s$ such that $Z(s)$ does not contain any of the points $z_i$, it suffices to show that
\[
\bigcap_i \left((k^{d+1})^\vee \smallsetminus \ker \ev_i\right)
\]
is non-empty, or, equivalently, that
\[
\bigcup_i \ker \ev_i
\]
is not all of $(k^{d+1})^\vee$. But because $\ker \ev_i$ has positive codimension, this follows from the fact that $|k| > n$.
\end{proof}

\begin{proof}
By the lemma, we are reduced to proving that $Z(\sigma_1 \oplus \sigma_2)$ consists of at most 16 closed points. But this follows from the result stated in the introduction, that summing over the closed points in $Z(\sigma_1 \oplus \sigma_2)$, with each point $z$ given weight $[\kappa(z) : k]$, gives 16. Because each such weight is $\geqslant 1$, there must be at most 16 points. 
\end{proof}

Returning to the Pl{\"u}cker embedding $X \to \mathbb{P}^9_k$, we have that there is a section $s$ of $\mathcal{O}_{\mathbb{P}^9_k}(1)$ such that $Z(s)$ is disjoint from $Z(\sigma_1 \oplus \sigma_2)$.

Now $s$ restricts on $X$ to a section of $\mathcal{O}_X(1)$. We define $U:= X \smallsetminus Z(s)$, and, because $Z(s)$ is disjoint from $Z(\sigma_1 \oplus \sigma_2)$, $U$ is a neighborhood of $Z(\sigma_1 \oplus \sigma_2)$. Moreover, $s$ is non-vanishing on $U$, and hence the assignment $\alpha \oplus \beta \mapsto s \otimes \alpha \oplus s^{\otimes 2} \otimes \beta$ determines an isomorphism
\[
\phi : E|_U \to \left(\mathcal{O}_X(1) \otimes \Sym^2(S^\vee) \oplus \mathcal{O}_X(2) \otimes \Sym^2(S^\vee)\right)|_U.
\]
Now let $\widetilde{E} := \mathcal{O}_X(1) \otimes \Sym^2(S^\vee) \oplus \mathcal{O}_X(2) \otimes \Sym^2(S^\vee)$, so that this isomorphism can be rewritten as $\phi : E|_U \to \widetilde{E}|_U$. We easily compute, again in the Picard group, that
\[
\det \widetilde{E} \otimes \omega_{X/k} = \mathcal{O}_X(10),
\]
whence $\widetilde{E}$ is relatively orientable. Moreover, chasing through the steps in the computation yields a canonical isomorphism $\rho : \det \widetilde{E} \otimes \omega_{X/k} \to (\mathcal{O}_X(5))^{\otimes 2}$. Through $\phi$, then, this also provides a relative orientation $\varrho$ of $E|_U$.

We are now prepared to prove the main result:

\begin{proof}[Proof of Main Result and Remark \ref{general}]

Consider the section
\[
\widetilde{\sigma} := \phi(\sigma) = s \otimes \sigma_1 \oplus s^{\otimes 2} \otimes \sigma_2.
\]

\noindent By Theorem 2.1 of \cite{debarremanivel}, we choose $f_1$ and $f_2$ general so that $Z(\sigma_1 \oplus \sigma_2)$ is finite {\'e}tale over $k$. In the case of a finite field, this may correspond to a finite extension of the base field; we will now denote this extension by $k$ (see Remark \ref{general}).  By construction (see Definition \ref{twjacform}), we have for each $L \in Z(\sigma_1 \oplus \sigma_2)$, that
\[
\Tr_{\kappa(L)/k} \langle \widetilde{\Jac}_z (f_1,f_2; s)\rangle := \Tr_{\kappa(L)/k} \langle \Jac_L (\widetilde{\sigma}; \rho)\rangle.
\]
Hence it suffices to show that
\[
\sum_{L \in Z(\sigma_1 \oplus \sigma_2)} \Tr_{\kappa(L)/k} \langle \Jac_L (\widetilde{\sigma}; \rho)\rangle = 8H.
\]
But by Proposition \ref{oriented index}, we have that 
\[
\sum_{L \in Z(\sigma_1 \oplus \sigma_2)} Tr_{\kappa(L)/k} \langle  \Jac_L(\widetilde{\sigma}; \rho) \rangle = \sum_{z \in Z(\widetilde{\sigma}) \cap U} \ind^{or}_z(\widetilde{\sigma}; \rho).
\]
We will consider both sides of this equation, and check two facts:
\begin{enumerate}[(i)]
\item The left side has rank 16.
\item The right side is an integral multiple of $H$.
\end{enumerate}

\noindent To check (i), first note that
\[
\rk \left( \sum_{L \in Z(\sigma_1 \oplus \sigma_2)} \Tr_{\kappa(L)/k} \langle \Jac_L (\widetilde{\sigma}; \rho)\rangle \right) = \sum_{L \in Z(\sigma_1 \oplus \sigma_2)} \rk \left(\Tr_{\kappa(L)/k} \langle \Jac_L (\widetilde{\sigma}; \rho) \rangle\right)
\]
Moreover, almost by construction (see Definition \ref{jacform}),
\[
\rk \Tr_{\kappa(L)/k} \langle \Jac_L (\widetilde{\sigma}; \rho) \rangle = [\kappa(L):k].
\]
Hence
\[
\rk \left( \sum_{L \in Z(\sigma_1 \oplus \sigma_2)} \Tr_{\kappa(L)/k} \langle \Jac_L (\widetilde{\sigma}; \rho)\rangle \right) = \sum_{L \in Z(\sigma_1 \oplus \sigma_2)} [\kappa(L):k],
\]
and we explained in the first section why the right side is equal to 16.

To check (ii), we first describe the structure of $Z(\widetilde{\sigma})$. Prima facie, it is given by
\[
Z(\widetilde{\sigma}) = Z(s \otimes \sigma_1) \cap Z(s^{\otimes 2} \otimes \sigma_2) = \left(Z(s) \coprod Z(\sigma_1)\right) \cap \left(Z(s^{\otimes 2}) \coprod Z(\sigma_2)\right).
\]
But because $Z(s)$ and $Z(s^{\otimes 2})$ are both disjoint from $Z(\sigma_1) \cap Z(\sigma_2) = Z(\sigma_1 \oplus \sigma_2)$ by assumption, this simplifies to
\[
Z(\widetilde{\sigma}) = Z(s) \cap Z(s^{\otimes 2}) \coprod Z(\sigma_1 \oplus \sigma_2).
\]
But $Z(s)  \cap Z(s^{\otimes 2} ) = Z(s)$, so we finally obtain
\[
Z(\widetilde{\sigma}) = Z(s) \coprod Z(\sigma_1 \oplus \sigma_2),
\]
expressing the zero scheme of $\widetilde{\sigma}$ as a the disjoint union\footnote{It is the appearance of $Z(s)$ as a component of the zero locus which motivates the appearance of the $s^{\oplus 2}$ factor in the second summand of $\widetilde{\sigma}$.} of $Z(s)$, which is regularly embedded because it is locally given by a regular sequence containing the single element $s$, and $Z(\sigma_1 \oplus \sigma_2)$, which is regularly embedded by assumption.

Hence $Z(\widetilde{\sigma})$ is regularly embedded, so now for each clopen component $Z_k$ of $Z(\widetilde{\sigma})$, let $\mathcal{E}_k$ denote the excess normal bundle on $Z_k$ described in Section \ref{oriented intersection}, let $j_k :Z_k \to X$ be the inclusion, and let $i : Z \to X$ be the inclusion of the whole zero locus. Recall that the oriented excess intersection formula (see Section \ref{oriented intersection}, particularly Formula 1 and the discussion preceding) computes
\[
\widetilde{e}(\widetilde{E})([X]) = \sum_k i_*(\widetilde{e}(\mathcal{E}_k)(j_k^*([X]))),
\]
where $\widetilde{e}$ is the Chow-Witt Euler class of Barge-Morel and Fasel (again see Section \ref{oriented intersection}).

Now letting $Z_0 = Z(s)$ and $Z_1, \ldots, Z_m$ denote the closed points making up $Z(\sigma_1\oplus \sigma_2)$, we have
\[
\sum_{k =1}^m i_*(\widetilde{e}(\mathcal{E}_k)(j_k^*([X]))) = \widetilde{e}(\widetilde{E})([X]) - i_*(\widetilde{e}(\mathcal{E}_0)(j_0^*([X]))),
\]
and hence (see Section \ref{oriented intersection} for notation)
\[
\sum_{z \in Z(\sigma_1 \oplus \sigma_2)} \ind_z^{or}(\widetilde{\sigma};\rho) = p^\rho_*(\widetilde{e}(\widetilde{E})([X])) - p^\rho_*(i_*(\widetilde{e}(\mathcal{E}_0)(j_0^*([X])))).
\]

Because $\widetilde{E}$ has an odd-rank summand, $p^\rho_*(\widetilde{E}([X]))$ is an integer multiple of $H$ by a result of Ananyevskiy (Theorem 7.4 of \cite{sloriented}). Moreover, because $Z_0 = Z(s)$ has codimension 1 in $X$, and $\dim X = 6$, we have that $\mathcal{E}_0$ is itself odd rank, so by the same result of Ananyevskiy, $p^\rho_*(i_*(\widetilde{e}(\mathcal{E}_0)(j_0^*([X]))))$ is also an integer multiple of $H$.

Thus we have that the sum
\[
\sum_{L \in Z(\sigma_1 \oplus \sigma_2)} \Tr_{\kappa(L)/k} \langle \widetilde{\Jac}_z (f_1,f_2; s)\rangle = \sum_{L \in Z(\sigma_1 \oplus \sigma_2)} \Tr_{\kappa(L)/k} \langle \Jac_L (\widetilde{\sigma}; \rho)\rangle
\]
is an integral multiple of $H$ in $GW(k)$, which has rank 16, and hence
\[
\sum_{L \in \Lines(\Sigma)} \Tr_{\kappa(L)/k} \langle \widetilde{\Jac}_z (f_1,f_2; s)\rangle = 8H.
\]

\end{proof}

\section{Extension of Main Result}

As we stated at the beginning, this result can also be extended to the case where $|k| \leqslant 16$. Recall that the reason for making the requirement $|k| > 16$ was to ensure that we could find a section of $\mathcal{O}_X(1)$ which did not vanish on $Z(\sigma_1 \oplus \sigma_2)$. Our strategy for extending the main result to smaller fields is to choose an appropriate extension of $k$ with enough elements. We first explain how to choose such an extension.

For any field extension $K/k$, extension of scalars induces a map $GW(k) \to GW(K)$. In the case where $K$ and $k$ are finite fields of odd characteristic, this map is strongly controlled by its behavior on rank-one forms, in the following sense:

\begin{prop}
Let $K/k$ be an extension of finite fields of odd characteristic. Suppose that $k \to K$ induces an isomorphism $k^\times/(k^\times)^2 \to K^\times/(K^\times)^2$. Then the map $GW(k) \to GW(K)$ induced by extension of scalars is an isomorphism.
\end{prop}

\noindent This result is well-known, but we include an outline of a proof for the sake of completeness:

\begin{proof}
It suffices to show that extension of scalars induces a bijection between isomorphism classes of non-degenerate quadratic forms over $k$ and over $K$. Let $Q$ be a non-degenerate quadratic form over $K$. By Theorem II.3.5 (in particular, the proof of part (1)) and Proposition II.3.4 of \cite{forms}, $Q$ can be expressed a diagonal matrix with entries all $1$'s except the last entry, $d(Q)$. Moreover $d(Q)$ is uniquely determined up to multiplication by a square, i.e. the class of $d(Q)$ in $K^\times/(K^\times)^2$ is well-defined, and depends only on the isomorphism class of $Q$. Finally, $Q$ is itself completely determined up to isomorphism by its rank and $d(Q)$.

The same applies to quadratic forms over $k$, and because extension of scalars is determined at the level of matrices by regarding a matrix over $k$ as a matrix over $K$, it is clear that extension of scalars sends a quadratic form $Q$ over $k$ of rank $r$ to a quadratic form $Q'$ over $K$ of rank $r$, with $d(Q')$ being the image in $K^\times/(K^\times)^2$ of $d(Q)$ under the map $k^\times/(k^\times)^2 \to K^\times/(K^\times)^2$. The claim follows immediately.
\end{proof}

Now our goal in selecting an appropriate extension $K$ of $k$ is to choose an extension with more than 16 elements, but one such that the induced map $k^\times/(k^\times)^2 \to K^\times/(K^\times)^2$ is an isomorphism, and hence the entire map $GW(k) \to GW(K)$, is an isomorphism. When $k$ is infinite, no extension need be chosen, so it suffices only to consider finite $k$. Moreover, we are only considering fields of characteristic not equal to 2, so it suffices to consider $k = \mathbb{F}_{p^q}$ for odd $p$. 

Now recall that there is an extension $\mathbb{F}_{p^{q'}}/\mathbb{F}_{p^q}$ precisely when $q|q'$, and in this case $[\mathbb{F}_{p^{q'}}: \mathbb{F}_{p^q} ] = q'/q$. Moreover, it is clearly always possible to find some $q'$ such that $p^{q'} > 16$ and such that $q'/q$ is odd. Hence our strategy comes down to the following fact:

\begin{prop}\label{GW iso}
Let $q'/q$ be odd. Then the inclusion $\mathbb{F}_{p^q} \to \mathbb{F}_{p^{q'}}$ induces an isomorphism 
\[
\mathbb{F}_{p^q}^\times/(\mathbb{F}_{p^q}^\times)^2 \to \mathbb{F}_{p^{q'}}^\times/(\mathbb{F}_{p^{q'}}^\times)^2
\]
\end{prop}

\noindent As before, this is standard, but we include the proof for completeness:

\begin{proof} To check injectivity, first suppose there is some $a \in \mathbb{F}_{p^q}^\times$ representing a non-trivial class in the kernel. Then $a$ does not have a square root in $\mathbb{F}_{p^q}$, but has a square root in $\mathbb{F}_{p^{q'}}$. But then the splitting field of $a$ would be an intermediate field between $\mathbb{F}_{p^q}$ and $\mathbb{F}_{p^{q'}}$ whose degree over $\mathbb{F}_{p^q}$ is 2, which is impossible because $[\mathbb{F}_{p^{q'}} : \mathbb{F}_{p^q} ]$ is odd.

To check surjectivity, recall that the multiplicative group $\mathbb{F}_{p^q}^\times$  and $\mathbb{F}_{p^{q'}}$ are cyclic groups of order $p^q - 1$ and $p^{q'} - 1$, respectively, and that the former is a subgroup of the latter. Hence, to check that an element of $\mathbb{F}_{p^{q'}}^\times$ is actually an element of $\mathbb{F}_{p^q}^\times$, it suffices to check that it has order dividing $p^q- 1$. 

Now consider any $a \in \mathbb{F}_{p^{q'}}^\times$. To show surjectivity, we will check that $a$ can be written as $a = \alpha \beta^{2}$ for some $\alpha \in \mathbb{F}_{p^q}$. To do this, set
\[
\ell := \frac{p^{q'} - 1}{p^q - 1} = 1 + (p^q) + \cdots + (p^q)^{q'/q - 1},
\]
and notice that $a^\ell$ has order dividing $p^q - 1$, and hence is in $\mathbb{F}_{p^q}^\times$. Moreover, because $q'/q$ is odd by assumption, $\ell$ is a sum of an odd number of odd numbers, and hence is odd, say $\ell = 2m + 1$. But now setting $\alpha = a^{\ell}$, and $\beta =a^{-m}$, we have that $\alpha \in \mathbb{F}_{p^q}$ and $a = \alpha \beta^2$.
\end{proof}

Now consider some field $k$, with characteristic not equal to 2, and with $|k| \leqslant 16$ (and hence $k$ is automatically perfect). After choosing an identification $k = \mathbb{F}_{p^q}$, there is a canonical way to choose some $q'$ such that $p^{q'} > 16$ and $q'/q$ is odd, namely we simply pick the smallest such $q'$. We now formalize this:

\begin{dfn}
Let $p \neq 2$ be prime, and $q$ some positive integer such that $p^q \leqslant 16$. Then $\hat{q}$ denotes the smallest multiple of $q$ such that $\hat{q}/q$ is odd and $p^{\hat{q}} > 16$.
\end{dfn}

This canonical choice also leads to a canonical extension $\mathbb{F}_{p^{\hat{q}}} / \mathbb{F}_{p^q}$, and thus after choosing an identification $k = \mathbb{F}_{p^q}$, there is a canonical extension $K/k$ such that $k \to K$ induces an isomorphism $GW(k) \to GW(K)$, namely we choose $K = \mathbb{F}_{p^{\hat{q}}}$. Note that the isomorphism $GW(k) \to GW(K)$ is now also canonical, because the inclusion $k \to K$ is canonical.

Now let $\hat{X}$ denote the base change of $X$ from $k$ to $K$, $\hat{f}_1$ and $\hat{f}_2$ the base changes of $f_1$ and $f_2$, and $\hat{\sigma}_1 \oplus \hat{\sigma}_2$ the base change of $\sigma_1 \oplus \sigma_2$. Then $Z(\hat{\sigma}_1 \oplus \hat{\sigma}_2)$ is itself the base change of $Z(\sigma_1 \oplus \sigma_2)$. Now let $\pi : Z(\hat{\sigma}_1 \oplus \hat{\sigma}_2) \to Z(\sigma_1 \oplus \sigma_2)$ denote the base change projection. Because $|K| > 16$, there is some section $\hat{s}$ of $\mathcal{O}_{\hat{X}}(1)$ whose zero locus misses $Z(\hat{\sigma}_1 \oplus \hat{\sigma}_2)$, and so we can define an extended version of the twisted Jacobian form:

\begin{dfn}[extension of twisted Jacobian form]
With the notation as in the preceding paragraphs, consider some $z \in Z(\sigma_1 \oplus \sigma_2)$. We define
\[
\widehat{\Tr}_{\kappa(z)/k} \langle \widetilde{\Jac}_z(f_1, f_2; \hat{s}) \rangle := \sum_{y \in \pi^{-1}(z)} \Tr_{\kappa(y)/K} \langle \widetilde{\Jac}_y (\hat{f}_1, \hat{f}_2; \hat{s})\rangle,
\]
which we regard as an element of $GW(k)$ through the identification $GW(k) \to GW(K)$. 
\end{dfn}

Considering our main result applied to $\hat{\Sigma} = Z(\hat{f}_1,\hat{f}_2)$, we have the following extension of our main result directly from definitions:

\begin{thm}[Extension of Main Result when $|k| \leqslant 16$]
Let $k = \mathbb{F}_{p^q}$, with $p \neq 2$, and consider a general del Pezzo surface $\Sigma = Z(f_1, f_2)$ in $\mathbb{P}^4_k$. Choose a section $\hat{s}$ of $\mathcal{O}_{\hat{X}}(1)$ which does not vanish on $Z(\hat{\sigma}_1 \oplus \hat{\sigma_2})$ (see the preceding paragraphs for notation). Then
\[
\sum_{L \in \Lines(\Sigma)} \widehat{\Tr}_{\kappa(L)/k} \langle \widetilde{\Jac}_z(f_1, f_2; \hat{s}) \rangle = 8H.
\]
\end{thm}

\bibliographystyle{plain}

\bibliography{bibliography.bib}{}

\end{document}